\documentclass[12pt,leqno]{article}

\usepackage[cp1250]{inputenc}
\usepackage[OT4]{fontenc}
\usepackage{amsthm}
\usepackage{amsfonts}
\usepackage{amsmath}
\usepackage{amssymb}

\newcommand{\R}{\mathbb{R}}

\newcommand{\bo}{{\bf 0}}

\newcommand{\rank}{\operatorname{rank}}
\newcommand{\sgn}{\operatorname{sgn}}

\newcommand{\Aa}{\mathcal{A}}

\newtheorem{theorem} {Theorem}[section]
\newtheorem{theorem*}{Theorem}
\newtheorem{prop*} {Proposition}
\newtheorem{lemma*}{Lemma}
\newtheorem{lemma}[theorem]{Lemma}
\newtheorem{cor}[theorem]{Corollary}
\newtheorem{cor*}{Corollary}
\newtheorem{prop}[theorem] {Proposition}
\newtheorem{definition}[theorem]{Definition}
\newtheorem{definition*}{Definition}
\newtheorem{rem}[theorem]{Remark}

\theoremstyle{definition}
\newtheorem{ex}[theorem]{Example}

\begin{document} 
 

\begin{center}
{\Large Mappings from $\R^3$ to $\R^3$\\[1em] and signs of swallowtails}\\[1em]
by Justyna Bobowik and Zbigniew Szafraniec
\end{center}

\section{Introduction}
Let $M$ be an oriented 3-manifold. According to \cite{boardman}, for a residual
set of mappings $f\in C^\infty(M,\R^3)$, the
3-jet $j^3f$ intersects transversely the Boardman strata $S_1$, $S_{1,1}$, $S_{1,1,1}$ in $J^3(M,\R^3)$.
If that is the case then $S_{1,1,1}(f)$ is a discrete subset of $M$.

According to Morin \cite{morin}, if $j^3f$ is also transversal to $S_{1_3}$ then $S_{1,1,1}(f)=S_{1_3}(f)$
and  any $p\in S_{1,1,1}(f)$ is a swallowtail point, so that there exists a well-oriented coordinate system $x,y,z$ centered at $p$,
and a coordinate system centered at $f(p)$ such that $f$ has the form
$f_\pm(x,y,z)=(\pm xy+x^2 z+x^4,y,z)$. Hence one may associate a sign with $p\in S_{1,1,1}(f)$. A geometric definition of the sign associated with a swallowtail was originally introduced by Goryunov in \cite{goryunov}.

In this paper we give a definition of a simple swallowtail point $p\in S_{1,1,1}(f)$, and define its
sign $I(f,p)$. 
We shall show (see Theorem \ref{klasyfikacja}) that
\begin{itemize}
\item if $p$ is a simple swallowtail then $j^3f$ intersects transversely  $S_{1_3}$ at $p$,
\item $I(f,p)=+1$ (resp. $-1$) if and only if $f$ has the form $f_+$ (resp. $f_-$).
\end{itemize}

If every $p\in S_{1,1,1}(f)$ is a simple swallowtail and $S_{1,1,1}(f)$ is finite, then numbers
$\#S_{1,1,1}(f)$, $\#\{p\in S_{1,1,1}(f)\ |\ I(f,p)=\pm 1\}$
are important invariants associated
with $f$. (We refer the reader to \cite{goryunov, carmenetal}, where the
first-order invariants of stable mappings in $C^\infty(M,\R^3)$ are classified.
Articles \cite{marartari, moyaballesteros} present classification of some families
of germs from $\R^3$ to $\R^3$.)
We shall show  how to compute these numbers 
in the case where $f:\R^3\rightarrow\R^3$ is a polynomial mapping, in terms of signatures
of quadratic forms.

It is proper to add that in \cite{krzyzanowskaszafraniec} there is presented a construction
of quadratic forms whose signatures determine the number of positive and negative
cusps of a generic polynomial mapping from the plane to the plane.

\section{Preliminaries}
In this section we have compiled some basic facts on singularities of
mappings between manifolds. The best reference here is \cite{golub}.

Let $X,Y$ be finite-dimensional smooth manifolds. If $p$ is a point in $X$
then $E(p)$ denotes the local ring of germs of smooth functions at $p$.
Its maximal ideal $m(p)$ consists of germs of functions vanishing at $p$.
A map germ $f:(X,p)\rightarrow (Y,q)$ induces a homomorphism of local
rings $f^*:E(q)\rightarrow E(p)$. The local ring of $f$ at $p$ is the quotient
ring ${\cal R}_f(p)=E(p)/E(p)f^*(m(q))$.

We say that $f$ takes on a Morin singularity of type $k$ at $p$
if ${\cal R}_f(p)\cong \R[t]/(t^{k+1})$ for some $k$.

Let $J^k(X,Y)$ be the $k$-jet bundle over $X\times Y$. 
According to \cite{morin} and \cite{levine1,levine2}, if $\dim X=\dim Y$ then there
exists a submanifold $S_{1_k}$ of $J^k(X,Y)$ of codimension $k$ such that
for a smooth $f:X\rightarrow Y$  the set $S_{1_k}(f)=(j^k f)^{-1}(S_{1_k})$ consists of points
where $f$ takes on a Morin singularity of type $k$. For a residual set of mappings
$f\in C^\infty(X,Y)$, $S_{1_k}(f)$ is a submanifold of $X$ of codimension $k$.

For a mapping $f:X\rightarrow Y$, a point $p$ is in $S_1(f)$ if $f$ has corank $1$ at $p$.
A point $p\in S_1(f)$ is in $S_{1,1}(f)$ if $f|S_1(f)$ has corank $1$, assuming that
$S_1(f)$ is a submanifold.
A point $p\in S_{1,1}(f)$ is in $S_{1,1,1}(f)$ if $f|S_{1,1}(f)$ has corank $1$,
assuming that $S_{1,1}(f)$ is a submanifold.

Acording to Boardman \cite{boardman}, for a residual set of mappings $f\in C^\infty(X,Y)$,
sets $S_1(f)$, $S_{1,1}(f)$, and $S_{1,1,1}(f)$ are submanifolds.

\section{Sign of a swallowtail}

Let $f_\pm=(\pm xy+x^2z+x^4,y,z):(\R^3,\bo)\rightarrow(\R^3,\bo)$. 
\begin{ex}\label{przyklad1}
There exists an orientation reversing diffeomorphism $\phi_-=(-x,y,z):(\R^3,\bo)\rightarrow(\R^3,\bo)$ such that $f_-=f_+\circ\phi_-$.
\end{ex}

\begin{ex}
There exist an orientation preserving diffeomorphism $\phi_+=(-x,-y,z):(\R^3,\bo)\rightarrow(\R^3,\bo)$ 
and an orientation reversing diffeomorphism $\psi_-=(x,-y,z):(\R^3,\bo)\rightarrow(\R^3,\bo)$ such that
$ f_+= \psi_-\circ f_+\circ\phi_+$.
\end{ex}

Let $f:(\R^3,p_0)\rightarrow(\R^3,f(p_0))$ be a smooth mapping defined in a neighbourhood
of $p_0$, and let $J=\det [Df]$. Denote  $S_1(f)=\{p\in\R^3\ |\ \rank Df(p)=2\}$.

\begin{rem} \label{remark1}
Assume that $\rank[Df(p_0)]=2$ and the gradient $\nabla J(p_0)\neq\bo$.
Then $J(p_0)=0$  and the set $J^{-1}(0)$ of critical points of $f$ is locally a smooth surface. Moreover $J^{-1}(0)=S_1(f)$
 near $p_0$. Hence $\dim \,\operatorname{Ker}\,Df(p)=1$
along $S_1(f)$.
\end{rem}

From now on we shall suppose that assumptions of the above remark hold.
Let $K:(\R^3,p_0)\rightarrow\R^3$ be a vector field such that $K(p_0)\neq\bo$
and $Df(p)\,K(p)\equiv \bo$ for $p\in S_1(f)$, so that $K$ is in the kernel
of $Df$ along $S_1(f)$.

\begin{lemma} If $f=(f_1,f_2,f_3)$ then at least one vector product
$\nabla f_i\times \nabla f_j$ satisfies above conditions.
\end{lemma}
\begin{proof} Since $\rank [Df(p_0)]=2$, two rows of the derivative matrix  are linearly independent.
Hence at least one $2\times 2$--minor is  non zero, and then there exist $1\leq i<j\leq 3$ with $(\nabla f_i\times\nabla f_j)(p_0)\neq\bo$.

Let $p\in S_1(f)=J^{-1}(0)$. Notice that one of the coordinates of the composition $Df(p)\, (\nabla f_i\times\nabla f_j)(p)$
equals $\pm J(p)=0$. The other two coordinates are equal to the Laplace expansion of matrices having the same two rows, so they vanish too.
So $Df(p)\, (\nabla f_i\times \nabla f_j)(p)\equiv\bo$.
\end{proof}

Put $X=\langle\nabla J,K\rangle$, where $\langle\,.\, ,\, .\,\rangle$ denotes the standard scalar product.
Denote  $S_{1,1}(f)=\{p\in S_1(f)\ |\ \rank D(f|S_1(f))(p)=1\}$.

\begin{rem}\label{remark2}
 As  $\dim \,\operatorname{Ker}\,Df=1$
along $S_1(f)$, then the set of critical points of $f|S_1(f)$ equals $S_{1,1}(f)$.
Suppose that $X(p_0)=0$ and vectors $\nabla J(p_0),\nabla X(p_0)$ are linearly independent.
Then locally 
$$S_{1,1}(f)=\{p\in S_1(f)\ |\ K(p)\in T_p S_1(f)\}=
\{p\in J^{-1}(0)\ |\ K(p)\in T_p J^{-1}(0)\}$$
$$=\{p\in J^{-1}(0)\ |\ K(p)\perp \nabla J(p)\}
=J^{-1}(0)\cap X^{-1}(0).$$ Moreover $S_{1,1}(f)$ is locally a smooth curve near $p_0$.
\end{rem}
From now on we shall suppose that assumptions of the above remark hold.
Put $Y=\langle \nabla X,K\rangle$. Denote
$S_{1,1,1}(f)=\{p\in S_{1,1}(f)\ |\ \rank D(f|S_{1,1}(f))(p)=0\}$.

\begin{rem}\label{remark3}
As  $S_{1,1}(f)$ is one-dimensional, then the set of critical points of $f|S_{1,1}(f)$ equals $S_{1,1,1}(f)$.
 Suppose that $Y(p_0)=0$.
Then locally
$$S_{1,1,1}(f)=\{p\in S_{1,1}(f)\ |\ K(p)\in T_p(S_{1,1}(f))\}$$
$$=\{p\in J^{-1}(0)\cap X^{-1}(0)\ |\ K(p)\in T_p(J^{-1}(0))\cap T_p(X^{-1}(0))\}$$
$$=\{p\in J^{-1}(0)\cap X^{-1}(0)\ |\ K(p)\perp\nabla J(p)\ ,\ K(p)\perp\nabla X(p)\}$$
$$=J^{-1}(0)\cap X^{-1}(0)\cap Y^{-1}(0).$$
\end{rem}

\begin{definition} We shall say that $p_0$ is a simple swallowtail point if  $\rank[Df(p_0)]=2$, 
 $J(p_0)=X(p_0)=Y(p_0)=0$, and vectors
$\nabla J(p_0), \nabla X(p_0), \nabla Y(p_0)$ are linearly independent. In that case we say that
the germ $f$ is a simple swallowtail
\end{definition}

If that is the case then the vectors $\nabla J(p_0),\nabla X(p_0),\nabla Y(p_0)$
are linearly independent, so  $\{p_0\}$ is an isolated point in $S_{1,1,1}(f)$.
Put $Z=\langle\nabla Y,K\rangle$. Since $X(p_0)=Y(p_0)=0$, the vector $K(p_0)$
is perpendicular to both $\nabla J(p_0)$ and $\nabla X(p_0)$.
Hence $Z(p_0)=\langle \nabla Y(p_0),K(p_0)\rangle\neq 0$.

\begin{cor}\label{star} Let $H=(J,X,Y)$.
The point $p_0$ is a simple swallowtail if and only if $\rank[Df(p_0)]=2$ and
$H$ has a simple zero at $p_0$. If that is the case then  $Z(p_0)\neq 0$. $\Box$
\end{cor}

\begin{definition} If the germ $f:(\R^3,p_0)\rightarrow(\R^3,f(p_0))$ is a simple swallowtail then its sign, denoted by  $I(f,p_0)$, is defined to be
$I(f,p_0)=\sgn(Z(p_0)\cdot\det[DH(p_0)])$, i.e. $$I(f,p_0)=\sgn(Z(p_0))\cdot\sgn(\det[\nabla J(p_0),\nabla X(p_0),\nabla Y(p_0)]).$$
\end{definition}

\begin{ex}\label{przyklad3}      
 Let  $f=f_\pm=(\pm xy+x^2z+x^4,y,z)$,  $p_0=\bo$. Then
\[   Df=\left[\begin{array}{ccc} \pm y+2xz+4x^3 & * & *\\
0  & 1  &   0\\    0  &   0   &   1          \end{array}     \right]\ ,    \]
$$J=\pm y+2xz+4x^3\ ,\ K=(1,0,0),$$
$$X=2z+12x^2\ ,\ Y=24x\ ,\ Z=24\ .$$
Because
\[  \det [DH(\bo)]=\det  \left[  \begin{array}{rrr}  0 & \pm 1 & 0\\ 0 & 0 & 2\\ 24 & 0 & 0     \end{array}    \right]  = \pm 48\ ,   \]
then  both $f_\pm$ are simple swallowtails and $I(f_\pm ,\bo)=\pm 1$.\\[1em]
In the next sections we shall show that\begin{itemize}
\item[(1)] $I(f,p_0)$ does not depend on the choice of the field $K$,
\item[(2)] if $\phi, \psi$ are germs of diffeomorphisms and $\phi$ preserves the orientation then
the sign of $f$ equals the sign of $\psi\circ f\circ\phi$,
\item[(3)] there exist germs of diffeomorphisms $\phi, \psi$ such that $\phi$ preserves the
orientation and $\psi\circ f\circ\phi=f_+$ (resp. $=f_-$) if $I(f,p_0)=+1$ (resp. $-1$).
\end{itemize}

\end{ex}

\section{On the choice of the field $K$}
In this section we shall show that the sign of a swallowtail does not depend on
the choice of the field $K$.

 Suppose that assumptions of Remark \ref{remark1} hold and
$K:(\R^3,p_0)\rightarrow\R^3$ is the vector field introduced in the previous section.

Let $K_1:(\R^3,p_0)\rightarrow\R^3$ be a vector field such 
and $Df(p)\,K_1(p)\equiv \bo$ for $p\in S_1(f)$, so that $K_1$ is in the kernel
of $Df$ along $S_1(f)$. (We do not exclude the case where $K_1(p_0)=\bo$.)
The kernel of $Df$ along $S_1(f)$ is one-dimensional, so fields $K$ and $K_1$
are collinear along $S_1(f)$, and there exists a smooth $\xi:(\R^3,p_0)\rightarrow\R$
such that $K_1(p)=\xi(p)\cdot K(p)$ for $p\in S_1(f)$. 
Since $S_1(f)=J^{-1}(0)$ is locally a complete intersection, there is $L:(\R^3,p_0)\rightarrow\R^3$
with $K_1=\xi\cdot K+J\cdot L$.

Take any smooth $\alpha:(\R^3,p_0)\rightarrow\R$, and set $J_1=\alpha\cdot J$.
Of course, $J_1(p_0)=0$. Put $X_1=\langle\nabla J_1,K_1\rangle$,
$Y_1=\langle\nabla X_1,K_1\rangle$. We have

$$X_1=\langle\nabla(\alpha\cdot J),K_1\rangle=\langle\alpha\cdot\nabla J+J\cdot\nabla\alpha\ ,\ \xi\cdot K+J\cdot L\rangle
=(\alpha\cdot\xi)\cdot X+J\cdot R_1\ ,$$  
$$Y_1=\langle\nabla((\alpha\cdot\xi)\cdot X)+\nabla(J\cdot R_1)\ ,\ K_1\rangle$$
$$=\langle  (\alpha\cdot\xi)\nabla X+X\cdot\nabla(\alpha\cdot\xi)+J\cdot\nabla(R_1)+R_1\cdot\nabla J\ ,\ \xi\cdot K+J\cdot L    \rangle$$
$$=(\alpha\cdot\xi^2)Y+R_2\cdot J+( \langle\nabla(\alpha\cdot\xi),K_1\rangle+R_1\cdot\xi            )\cdot X$$
$$=(\alpha\cdot\xi^2)Y+R_2\cdot J+R_3\cdot X\ ,$$
where $R_1,R_2,R_3:(\R^3,p_0)\rightarrow\R$. Hence $X_1(p_0)=0$, $Y_1(p_0)=0$, and
\[ \left[\begin{array}{c} J_1\\ X_1\\ Y_1  \end{array}   \right]    
\ =\ \left[ \begin{array}{ccc}\alpha & 0 & 0\\ R_1 & \alpha\cdot\xi & 0\\ R_2 & R_3 & \alpha\cdot \xi^2      \end{array}      \right]
 \left[\begin{array}{c}  J \\ X \\ Y    \end{array}    \right] .
\]

\begin{cor}\label{wwniosek1} One has $J^{-1}(0)\subset J_1^{-1}(0)$, $J^{-1}(0)\cap X^{-1}(0)\subset J_1^{-1}(0)\cap X_1^{-1}(0)$, and
$J^{-1}(0)\cap X^{-1}(0)\cap Y^{-1}(0)\subset J_1^{-1}(0)\cap X_1^{-1}(0)\cap Y_1^{-1}(0)$. $\Box$
\end{cor}

Put $H_1=(J_1,X_1,Y_1):(\R^3,p_0)\rightarrow (\R^3,\bo)$. 

\begin{lemma}
We have $H^{-1}(\bo)\subset H_1^{-1}(\bo)$ near $p_0$. If $\alpha(p_0)\xi(p_0)\neq 0$ then locally
$H^{-1}(\bo)=H_1^{-1}(\bo)$. $\Box$
\end{lemma}

Let $Z_1=\langle\nabla Y_1,K_1\rangle$.

\begin{lemma}\label{lemat1}
We have $Z_1(p_0)=\alpha(p_0)\xi^3(p_0)Z(p_0)$. 
\end{lemma}
\noindent{\em Proof.} Since $J(p_0)=X(p_0)=Y(p_0)=0$, one has
$$\nabla Y_1(p_0)=\alpha(p_0)\xi^2(p_0)\nabla Y(p_0)+R_2(p_0)\nabla J(p_0)+R_3(p_0)\nabla X(p_0).$$
Then $Z_1(p_0)=\langle\nabla Y_1(p_0),\xi(p_0)K(p_0)\rangle
=\alpha(p_0)\xi^3(p_0)Z(p_0)+R_2(p_0)\xi(p_0)X(p_0)$ $+R_3(p_0)\xi(p_0)Y(p_0)=\alpha(p_0)\xi^3(p_0)Z(p_0).\ \Box$

\begin{prop}\label{lemat2}
We have
\begin{itemize}
\item[(i)] $\nabla J_1(p_0)$ is a linear combination of $\nabla J(p_0)$,
$\nabla X_1(p_0)$ is a linear combination of $\nabla J(p_0),\nabla X(p_0)$, and $\nabla Y_1(p_0)$
is a linear combination of $\nabla J(p_0),\nabla X(p_0),\nabla Y(p_0)$,
\item[(ii)] if $\alpha(p_0)\xi(p_0)\neq 0$ then $p_0$ is a simple zero of $H$ if and only if $p_0$ is a simple zero of $H_1$. If that is the case then
$$\sgn(Z_1(p_0)\cdot \det[DH_1(p_0)])=\sgn(Z(p_0)\cdot\det[DH(p_0)]    )=I(f,p_0).$$
In particular, the definition of $I(f,p_0)$ does not depend on the choice of the vector field $K$ with $K(p_0)\neq\bo$. (In this case one may
take $\alpha\equiv 1$.)
\item[(iii)] if $\alpha(p_0)\xi(p_0)=0$ then $\sgn(Z_1(p_0)\cdot\det[DH_1(p_0)]  )=0$.
\end{itemize}
\end{prop}
\noindent{\em Proof.} We have
$$\nabla J_1(p_0)=\alpha(p_0)\nabla J(p_0),$$ 
$$\nabla X_1(p_0)=R_1(p_0)\nabla J(p_0)+       \alpha(p_0)\xi(p_0)\nabla X(p_0),$$
$$\nabla Y_1(p_0)=R_2(p_0)\nabla J(p_0)+R_3(p_0)\nabla X(p_0)+\alpha(p_0)\xi^2(p_0)\nabla Y(p_0).$$

Then $\det[DH_1(p_0)]=\det[\nabla J_1(p_0),\nabla X_1(p_0),\nabla Y_1(p_0)]$
$$= \alpha^3(p_0)\xi^3(p_0)\det[\nabla J(p_0),\nabla X(p_0),\nabla Y(p_0)]= \alpha^3(p_0)\xi^3(p_0)\det[DH(p_0)].$$

By Lemma \ref{lemat1}, if $\alpha(p_0)\xi(p_0)\neq 0$ then
$$\sgn( Z_1(p_0)\cdot \det[DH_1(p_0)]  )=\sgn(\alpha(p_0)\xi^3(p_0)Z(p_0)\cdot \alpha^3(p_0)\xi^3(p_0)\det[DH(p_0)]      )$$
$$=\sgn( Z(p_0)\cdot\det[DH(p_0)]  ).$$

If $\alpha(p_0)\xi(p_0)=0$ then $Z_1(p_0)\cdot\det[DH_1(p_0)]=0$. $\Box$\\[1em]

\section{Coordinates in the domain}

In this section we shall show that an orientation preserving change of coordinates in the domain
does not change the sign of a swallowtail.

 Suppose that the germ $f:(\R^3,p_0)\rightarrow(\R^3,f(p_0))$ is a simple swallowtail. 
Let $\phi:(\R^3,q_0)\rightarrow(\R^3,p_0)$ be a germ of a diffeomorphism.
Put $\tilde f=f\circ\phi$. Then the derivative matrix $[D\tilde f]=[Df(\phi)]\cdot[D\phi]$,
and then $\tilde J=\det[D\tilde f]=J(\phi)\cdot\det[D\phi]$.
Hence $J(\phi)=\tilde\alpha\cdot\tilde J$, where $\tilde\alpha=(\det[D\phi])^{-1}$,
$\tilde\alpha(q_0)\neq 0$.

We have $\nabla(J(\phi))=[D\phi]^T\cdot\nabla J(\phi)$ and
$\nabla J(\phi)(q_0)=\nabla J(p_0)\neq \bo$. 
Therefore $\nabla(J(\phi))(q_0)\neq \bo$.
Since $\tilde J(q_0)=0$, we have $\nabla(J(\phi))(q_0)=\tilde\alpha(q_0)\cdot\nabla\tilde J(q_0)$,
and then $\nabla\tilde J(q_0)\neq \bo$.
 
Denote $S_1(\tilde f)=\{q\in\R^n\ |\ \rank D\tilde f(q)=2\}$.
Since $\rank D\tilde f=\rank Df(\phi)$,
the set $\tilde J^{-1}(0)$   consisting of critical points of $\tilde f$ equals $\phi^{-1}(S_1(f))$, and so is a smooth surface.

Put $\tilde K=[D\phi]^{-1}\cdot K(\phi)$. Then
$$D\tilde f\cdot \tilde K=Df(\phi)\cdot D\phi\cdot [D\phi]^{-1}K(\phi)=(Df\cdot K)(\phi)\equiv \bo$$
along $S_1(\tilde f)$. Moreover $\tilde K(q_0)\neq\bo$.

Put $\tilde X=\langle\nabla\tilde J,\tilde K\rangle$ and $\tilde X_1=\langle\nabla(J(\phi)),\tilde K\rangle$. Then
$$ \tilde X_1=\langle [D\phi]^T\cdot\nabla J(\phi),[D\phi]^{-1}\cdot K(\phi)\rangle=
\langle\nabla J(\phi),K(\phi)\rangle=X(\phi)\ .      $$
Put $\tilde Y=\langle\nabla\tilde X,\tilde K\rangle$ and $\tilde Y_1=\langle\nabla\tilde X_1,\tilde K\rangle$.
Then 
$$\tilde Y_1=\langle\nabla(X(\phi)),\tilde K\rangle=\langle [D\phi]^T\cdot\nabla X(\phi),[D\phi]^{-1}\cdot K(\phi)\rangle
=\langle\nabla X(\phi),K(\phi)\rangle=Y(\phi)\ .$$

Take mappings $\tilde H=(\tilde J,\tilde X,\tilde Y)$ and $\tilde H_1=(J(\phi),X(\phi),Y(\phi))=H(\phi)$. 

Put $\tilde Z=\langle\nabla\tilde Y,\tilde K\rangle$ and $\tilde Z_1=
\langle\nabla\tilde Y_1,\tilde K\rangle$.
Then $\tilde Z_1=\langle\nabla(Y(\phi)),\tilde K\rangle=  \langle [D\phi]^T\cdot\nabla J(\phi),[D\phi]^{-1}\cdot K(\phi)\rangle= Z(\phi)$, so that $\tilde Z_1(q_0)=Z(p_0)\neq 0$.

We have $D\tilde H_1=DH(\phi)\cdot D\phi$, so that $\tilde H_1$ has a simple zero at $q_0$.
We may apply results of the previous section, where $\xi\equiv 1$, and we replace
$J$ by $\tilde J$, $\alpha$ by $\tilde\alpha$, $K$ by $\tilde K$, $X$ by $\tilde X$, $X_1$ by $\tilde X_1$, and so on.
By Proposition \ref{lemat2} (ii) and Corollary  \ref{star},  $\tilde H$ has a simple zero at $q_0$,  $\tilde Z(q_0)\neq 0$,   and
$$I(\tilde f,q_0)=\sgn (\tilde Z(q_0)\cdot\det [D\tilde H(q_0)])=\sgn(\tilde Z_1(q_0)\cdot\det [D\tilde H_1(q_0)])$$
$$=\sgn(Z(p_0)\cdot\det[DH(p_0)]\cdot \det[D\phi(q_0)])=I(f,p_0)\cdot\sgn\det [D\phi(q_0)].$$
We have got

\begin{prop}\label{fakt51}
If $\phi:(\R^3,q_0)\rightarrow(\R^3,p_0)$ is a germ of a diffeomorphism then
$f$ is a simple swallowtail if and only if $f\circ \phi$ does. If that is the case then
$$I(f\circ\phi,q_0)=I(f,p_0)\cdot\sgn\det[D\phi(q_0)]\ . \ \Box$$
\end{prop}

\section{Coordinates in the target}

In this section we shall show that a change of coordinates in the target
does not change the sign of a swallowtail.

Let $\psi:(\R^3,f(p_0))\rightarrow(\R^3,s_0)$ be a germ of a diffeomorphism.
Put $\bar f=\psi\circ f$.
Then the derivative matrix $[D\bar f]=[D\psi(f)]\cdot[Df]$, and then
$\bar J=\det[D\bar f]=\det[D\psi(f)]\cdot J=\bar\alpha\cdot J$, where
$\bar\alpha=\det[D\psi(f)]$, $\bar\alpha(p_0)\neq 0$.
Since $J(p_0)=0$, we have $\bar J(p_0)=0$ and $\nabla\bar J(p_0)=\bar\alpha(p_0)\nabla J(p_0)\neq\bo$.

Denote $S_1(\bar f)=\{p\in\R^n\ |\ \rank D\bar f(p)=2\}$.
Since $\rank D\bar f=\rank Df$,
the set $\bar J^{-1}(0)$   of critical points of $\bar f$ equals $S_1(f)$, and so is a smooth surface.
Of course $D\bar f\cdot K=D\psi(f)\cdot (Df\cdot K)\equiv\bo$ along $S_1(\bar f)=S_1(f)$.

Put $\bar X=\langle\nabla\bar J,K\rangle$. Hence
$ \bar X=\langle\bar\alpha\cdot\nabla J+J\cdot\nabla\bar\alpha,K\rangle=
\bar\alpha\cdot X+\bar R_1\cdot J$, where $\bar R_1=\langle\nabla\bar\alpha,K\rangle$.
In particular, $\bar X(p_0)=0$.
Moreover, $\nabla\bar X=\bar\alpha\cdot\nabla X+X\cdot\nabla\bar\alpha+\bar R_1\cdot \nabla J+J\cdot\nabla\bar R_1$,
so that $\nabla \bar X(p_0)=\bar\alpha(p_0)\nabla X(p_0)+\bar R_1(p_0)\nabla J(p_0)$.

Put 
$\bar Y=\langle\nabla\bar X,K\rangle
=\langle\bar\alpha\cdot\nabla X+X\cdot\nabla\bar\alpha+\bar R_1\cdot\nabla J+J\cdot\nabla \bar R_1,K\rangle$
$$=\langle\nabla\bar R_1,K\rangle J+(\langle\nabla\bar\alpha,K\rangle+\bar R_1)\cdot X+\bar\alpha\cdot Y
=\bar R_2\cdot J+\bar R_3\cdot X+\bar\alpha\cdot Y\ .$$
In particular, $\bar Y(p_0)=0$. Moreover
$$\nabla\bar Y(p_0)=\bar R_2(p_0)\nabla J(p_0)+\bar R_3(p_0)\cdot \nabla X(p_0)+\bar\alpha(p_0)\nabla Y(p_0)\ .$$
Since $\bar\alpha(p_0)\neq 0$, vectors $\nabla\bar J(p_0)$, $\nabla\bar X(p_0)$, $\nabla \bar Y(p_0)$
are linearly independent.

Put  $\bar H=(\bar J,\bar X,\bar Y)$ and  $\bar Z=\langle\nabla\bar Y,K\rangle$. 
The mapping $\bar H$ has a simple zero at $p_0$ if and only if $H$ does. 
\begin{prop}\label{fakt61}
Let $\psi:(\R^3,f(p_0))\rightarrow(\R^3,s_0)$ be a germ of a diffeomorphism.
Then $f$ is a simple swallowtail if and only if $\psi\circ f$ does. If that is the case then
$$I(\psi\circ f,p_0)=I(f,p_0).$$
\end{prop}
\noindent{\em Proof.} We have
$\bar Z(p_0)=\langle \nabla\bar Y(p_0),K(p_0)\rangle=\bar R_2(p_0)\cdot X(p_0)+\bar R_3(p_0)\cdot Y(p_0)+\bar\alpha(p_0)\cdot Z(p_0)=\bar\alpha(p_0)\cdot Z(p_0) .$
Using similar arguments as before one may observe that
$$I(\psi\circ f,p_0)=\sgn(\bar Z(p_0)\cdot\det[D\bar H(p_0)])$$
$$=\sgn(\bar\alpha(p_0)\cdot Z(p_0)\cdot(\bar\alpha(p_0))^3\cdot\det[DH(p_0)])=I(f,p_0).\ \Box$$

By Propositions \ref{fakt51} and \ref{fakt61} we get

\begin{theorem}\label{ttt1}
Let $\phi:(\R^3,q_0)\rightarrow (\R^3,p_0)$ and $\psi:(\R^3,f(p_0))\rightarrow (\R^3,s_0)$ be germs
of diffeomorphisms. Then $f$ is a simple swallowtail if and only if $\psi\circ f\circ \phi$ does. If that is the case
then
$$I(\psi\circ f\circ \phi,q_0)=I(f,p_0)\cdot \sgn\det[D\phi(q_0)].\ \Box$$
\end{theorem}

\section{Normal form of a swallowtail}
In this section we show that there is a one-to-one correspondence between the sign of a swallowtail
and its normal form.
\begin{theorem}\label{klasyfikacja}
Assume that $f:(\R^3,p_0)\rightarrow(\R^3,f(p_0))$ is a simple swallowtail.
Then  $f$ takes on an $S_{1_3}$ singularity transversely at $p_0$.
Moreover there exists an orientation preserving diffeomorphism $\phi:(\R^3,\bo)\rightarrow(\R^3,p_0)$
and a diffeomorphism $\psi:(\R^3,f(p_0))\rightarrow (\R^3,\bo)$ such that
$\psi\circ f\circ \phi=f_+$ (resp. $f_-$) if and only if $I(f,p_0)=+1$ (resp. $-1$).
\end{theorem}
{\em Proof.} One may assume that $p_0=f(p_0)=\bo$. Since $\rank [Df(\bo)]=2$,
there exist coordinate systems centered at $\bo$ such that $f$ has the form
$f=(h(x,y,z),y,z)$, and $\frac{\partial h}{\partial x}(\bo)=0$.

Then $J=\frac{\partial h}{\partial x}$, and one may take $K=(1,0,0)$. Therefore
$$X=\frac{\partial^2h}{\partial x^2}\ \ ,\ \ Y=\frac{\partial^3h}{\partial x^3}\ .$$
As the mapping $f$ is a simple swallowtail at the origin, then
$$\frac{\partial h}{\partial x}(\bo)=\frac{\partial^2 h}{\partial x^2}(\bo)=\frac{\partial^3h}{\partial x^3}(\bo)=0\ ,$$
and the vectors $\nabla\left(\frac{\partial h}{\partial x}\right)(\bo)$, $\nabla\left(\frac{\partial^2 h}{\partial x^2}\right)(\bo)$,
$\nabla\left(\frac{\partial^3 h}{\partial x^3}\right)(\bo)$ are linearly independent.

By \cite[page 176]{golub}, the mapping $f$ takes on an $S_{1_3}$ singularity transversely at $\bo$.
By \cite{morin} or \cite[Theorem 4.1]{golub}, there are diffeomorphisms
$\phi,\psi:(\R^3,\bo)\rightarrow(\R^3,\bo)$ such that $\psi\circ f\circ\phi=f_+$.

By Example \ref{przyklad1}, there  is an orientation reversing diffeomorphism $\phi_-$ such that $f_+\circ\phi_-=f_-$.
Hence, if $\phi$ does reverse the orientation then $\phi\circ\phi_-$ does preserve it, and $\psi\circ f\circ(\phi\circ\phi_-)=f_-$.

Then we may assume that $\psi\circ f\circ\phi=f_{\pm}$ and $\phi$ preserves the orientation.
By Example \ref{przyklad3} and Theorem \ref{ttt1}
$$\pm 1=I(f_{\pm},\bo)=I(\psi\circ f\circ\phi,\bo)=I(f,\bo).$$
Hence $I(f,\bo)=+1$ if and only if $\psi\circ f\circ\phi=f_+$, and $I(f,\bo)=-1$ if and only if $\psi\circ f\circ\phi=f_-.\  \Box$

As an immediate consequence of Theorem \ref{ttt1} we get

\begin{theorem}\label{ttt2}
Assume  that $f:(\R^3,p_0)\rightarrow (\R^3,f(p_0))$, $g:(\R^3,p_1)\rightarrow (\R^3,g(p_1))$ are simple swallowtails.

Then $I(f,p_0)=I(g,p_1)$ (resp. $I(f,p_0)=-I(g,p_1)$) if and only if there exists an orientaion preserving (resp. reversing)
diffeomorphism $\phi:(\R^3,p_1)\rightarrow (\R^3,p_0)$ and a diffeomorphism $\psi:(\R^3,f(p_0))\rightarrow (\R^3, g(p_1))$
such that $\psi\circ f\circ \phi=g$.\ $\Box$
\end{theorem}

\section{Polynomial mappings $\R^3\rightarrow\R^3$}
First we recall a method for counting the number of real zeros of an ideal in the multivariate case.

Let $I\subset\R[x_1,\ldots,x_n]$ be an ideal such that $\Aa=\R[x_1,\ldots,x_n]/I$
is an algebra of finite dimension, so that the set $V(I)$ of real zeros of $I$ is finite.
For $u\in\R[x_1,\ldots,x_n]$ denote by $t(u)$ the trace of the linear endomorphism
$\Aa\ni a\mapsto u\cdot a\in\Aa$. Then $t:\Aa\rightarrow\R$ is a linear functional.

\begin{theorem}[\cite{becker, pedersenetal}]\label{traceformula}
Take $g\in\R[x_1,\ldots,x_n]$. Let $\Theta$ (resp. $\Psi$) be the quadratic form
on $\Aa$ given by $\Theta(h)=t(h^2)$ (resp. $\Psi(h)=t(g\cdot h^2)$), where $h\in\Aa$. Then
$$ \sigma(\Theta)=\# V(I),$$
$$\sigma(\Psi)=\sum\sgn(g(p)),\ \mbox{where}\ p\in V(I),$$
and $\sigma(\cdot)$ denotes the signature of a quadratic form.

Moreover, if $\Psi$ is non-degenerate then $g(p)\neq 0$ at each $p\in V(I)$.
If that is the case then
$$(\sigma(\Theta)+\sigma(\Psi))/2=\sum\sgn(g(p)),\ \mbox{where}\ p\in V(I)\cap\{g>0\},$$
$$(\sigma(\Theta)-\sigma(\Psi))/2=\sum\sgn(g(p)),\ \mbox{where}\ p\in V(I)\cap\{g<0\}.$$
\end{theorem}

In the remainder of this section we shall show how to apply the above result so as to
compute the number of positive
and negative swallowtails of an polynomial mapping $f=(f_1,f_2,f_3):\R^3\rightarrow\R^3$.

Let $J=\det[Df]$, and let $I_1$ be the ideal in $\R[x,y,z]$ generated by $J$, $\partial J/\partial x$, $\partial J/\partial y$, $\partial J/\partial z$.
Then $J^{-1}(0)$ is the set of critical points of $f$.
\begin{prop}\label{kryt}
Assume that $I_1=\R[x,y,z]$. Then $\nabla J(p)\neq 0$ at each $p\in J^{-1}(0)$, so that $J^{-1}(0)$ is either empty or
a smooth surface.
\end{prop}
\begin{proof}
Because $1\in I_1$, then $J$ and $\nabla J$ do not vanish simultaneously at any point.
\end{proof}

Let $K_{ij}=\nabla f_i\times \nabla f_j$, where $1\leq i<j\leq 3$, and let
$X_{ij}=\langle\nabla J,K_{ij}\rangle$. Let $I_2\subset\R[x,y,z]$ denote the ideal generated by $J$
and all $2\times 2$--minors of the derivative matrix $Df$. 
\begin{prop}\label{krytt}
Assume that $I_1=I_2=\R[x,y,z]$. Then
\begin{itemize}
\item[(i)] $\rank [Df(p)]=2$ at any $p\in J^{-1}(0)$, and $J^{-1}(0)=S_1(f)$ is either empty or a smooth surface,
\item[(ii)] at each $p\in J^{-1}(0)$, at least one $K_{ij}(p)\neq\bo$,

\end{itemize}
\end{prop}
\begin{proof}  By Proposition \ref{kryt}, the set $J^{-1}(0)$ is a smooth surface. Take $p\in J^{-1}(0)$.
Since $I_2=\R[x,y,z]$, there is at least one non--zero $2\times 2$--minor of $Df(p)$. Hence $\operatorname{rank}[Df(p)]=2$,
i.e. $p\in S_1(f)$, and at least one vector field $K_{ij}$ does not vanish at $p$.\end{proof}
Let $I_3$ denote the ideal generated by
$J,X_{1,2},X_{1,3},X_{2,3}$ and all $2\times 2$ --minors of the derivative matrix $D(J,X_{1,2},X_{1,3},X_{2,3})$.
\begin{prop}\label{propozycja84}
Assume that $I_1=I_2=I_3=\R[x,y,z]$. Then
\item[(i)] $S_{1,1}(f)=J^{-1}(0)\cap\bigcap X_{ij}^{-1}(0)$ is either empty or a smooth curve,
\item[(ii)] the set of critical points of $f|S_{1,1}(f)$ equals $S_{1,1,1}(f)$.
\end{prop}
\begin{proof}
 Take any $p\in J^{-1}(0)$. One may assume that $K_{1,2}(p)\neq\bo$.
Every $K_{ij}\in \operatorname{Ker}Df$ and $\dim\operatorname{Ker}(Df)\equiv 1$ along $S_1(f)$.
Therefore we may apply results of Section 3 and  Section 4, where $K=K_{1,2}$, $K_1=K_{ij}$, $X_1=X_{ij}$, $\alpha\equiv 1$, so that $J_1=J$.
By Corollary \ref{wwniosek1},
$$J^{-1}(0)\cap X_{1,2}^{-1}(0)\subset J^{-1}(0)\cap\bigcap_{1\leq i<j\leq 3}X_{ij}^{-1}(0).$$
The opposite inclusion is obvious. Therefore
$J^{-1}(0)\cap\bigcap X_{ij}^{-1}(0)=J^{-1}(0)\cap X_{1,2}^{-1}(0)$ near $p\in J^{-1}(0)$. 

 Suppose  that $p\in J^{-1}(0)\cap\bigcap X_{ij}^{-1}(0)$, and $K_{1,2}(p)\neq\bo$.
As $I_3=\R[x,y,z]$, then $$\rank D(J,X_{1,2},X_{1,3},X_{2,3})(p)\geq 2.$$ 
According to Lemma \ref{lemat2}(i),
each $\nabla X_{ij}(p)$ is a linear combination of $\nabla J(p),\nabla X_{1,2}(p)$.
Hence $\rank D(J,X_{1,2},X_{1,3},X_{2,3})(p)=\rank D(J,X_{1,2})(p)=2$.
By Remark \ref{remark2}, $S_{1,1}(f)=J^{-1}(0)\cap X_{1,2}^{-1}(0)=J^{-1}(0)\cap\bigcap X_{ij}^{-1}(0)$ is a smooth curve
near any $p\in J^{-1}(0)$. 
By Remark \ref{remark3}, the set of critical points of $f|S_{1,1}(f)$ equals $S_{1,1,1}(f)$.
\end{proof}

Let $Y_{ij}=\langle\nabla X_{ij},K_{ij}\rangle$,
$H_{ij}=(J,X_{ij},Y_{ij})$, $Z_{ij}=\langle \nabla Y_{ij},K_{ij}\rangle$, and $g_{ij}=Z_{ij}\cdot\det[DH_{ij}]$.
Then  $J,X_{ij},Y_{ij},Z_{ij},g_{ij}\in\R[x,y,z]$.

Let $I$ denote the ideal in $\R[x,y,z]$ generated by $J$, all $X_{ij}$, and all $Y_{ij}$. Let $\Aa=\R[x,y,z]/I$.
Put $g=\sum \alpha_{ij}\cdot g_{ij}$, where all $\alpha_{ij}$ are non--negative, and at least one is positive.
If $\dim_{\R}\Aa<\infty$ then the quadratic forms $\Theta,\Psi$ on $\Aa$ given by
$$\Theta(h)=t(h^2)\ \ \ ,\ \ \ \Psi(h)=t(g\cdot h^2)$$
are defined, as well as their signatures.

\begin{theorem}\label{liczenie}
Assume that $I_1=I_2=I_3=\R[x,y,z]$ and $\dim_{\R}\Aa<\infty$. Then
$S_{1,1,1}(f)=V(I)$ is finite, and
$$\sigma(\Theta)=\#\, S_{1,1,1}(f)\ .$$

If the quadratic form $\Psi$ is non--degenerate, then all points in $S_{1,1,1}(f)$ are
simple swallowtails and
$$\sigma(\Psi)=\sum I(f,p),\ \mbox{where}\ p\in S_{1,1,1}(f).$$
If that is the case then
$$\# \{p\in S_{1,1,1}(f)\ |\ I(f,p)=+1\}=\left( \sigma(\Theta)+\sigma(\Psi)    \right)/2,$$
$$\# \{p\in S_{1,1,1}(f)\ |\ I(f,p)=-1\}=\left( \sigma(\Theta)-\sigma(\Psi)    \right)/2.$$
\end{theorem}
\begin{proof}
Take $p\in S_{1,1}(f)$. By Proposition \ref{krytt} and \ref{propozycja84}, we may assume that $K_{1,2}(p)\neq\bo$ and $p\in J^{-1}(0)\cap\bigcap X_{ij}^{-1}(0)$.
By Remark \ref{remark3} , the mapping $f|S_{1,1}(f)$ has a critical point at $p$ if and only if $p\in J^{-1}(0)\cap X_{1,2}^{-1}(0)\cap Y_{1,2}^{-1}(0)$.
Applying Corollary \ref{wwniosek1} and similar arguments as in the proof of Proposition \ref{propozycja84},
one may show that $J^{-1}(0)\cap X_{1,2}^{-1}(0)\cap Y_{1,2}^{-1}(0)=   J^{-1}(0)\cap\bigcap X_{ij}^{-1}(0)\cap\bigcap Y_{ij}^{-1}(0)$ near $p$.

Hence the set $S_{1,1,1}(f)$ of critical points of $f|S_{1,1}(f)$ equals $V(I)$. As $\dim\Aa<\infty$, the set $V(I)$ is finite, and by Theorem \ref{traceformula}
$$\#S_{1,1,1}(f)=\# V(I)=\sigma(\Theta).$$

Suppose that $\Psi$ is non--degenerate. Take any $p\in V(I)$. By Theorem \ref{traceformula}, at least one $\alpha_{ij}\cdot g_{ij}=\alpha_{ij}\cdot Z_{ij}\cdot\det[DH_{ij}]$ does not vanish at $p$.
By Corollary \ref{star}, $p$ is a simple swallowtail. By Proposition \ref{lemat2}, each $g_{ij}(p)$ is either zero or its sign equals $I(f,p)$, so that
$g(p)$ has the same sign as $I(f,p)$. Thus
$$\sigma(\Psi)=\sum\sgn(g(p))=\sum I(f,p),$$
where $ p\in S_{1,1,1}(f)$.
The last assertion is obvious.
\end{proof}
\begin{ex} Let $f=(-x^2y+z,y^2+x,x^2yz+z^2+y)$. Applying methods presented in Theorem \ref{liczenie}, with the help of {\sc Singular} \cite{greueletal},
one may check that $I_1=I_2=I_3=\R[x,y,z]$, $\dim_{\R}\Aa=27$, $\sigma(\Theta)=1$, the quadratic form $\Psi$ is non-degenerate and
$\sigma(\Psi)=-1$. Hence $S_{1,1,1}(f)$ consists of one negative swallowtail.
\end{ex}
Take $u\in\R[x,y,z]$. Let $\Phi_1,\Phi_2$ denote quadratic forms on $\Aa$ given by
$$\Phi_1(h)=t(u\cdot h^2)\ \ \ ,\ \ \ \Phi_2(h)=t(u\cdot g\cdot h^2).$$
Applying the method for counting real points inside a real--algebraic constraint region presented in \cite{pedersenetal} one gets

\begin{theorem}\label{liczenie1}
Assume that $I_1=I_2=I_3=\R[x,y,z]$ and $\dim_{\R}\Aa<\infty$. 

If the quadratic forms $\Psi,\Phi_1,\Phi_2$ are non--degenerate, then all points in $S_{1,1,1}(f)$ are
simple swallowtails and
$$\# \{p\in S_{1,1,1}(f)\ |\ u(p)>0  \}=\left(    \sigma(\Theta)+\sigma(\Phi_1)      \right)/2,$$
$$\# \{p\in S_{1,1,1}(f)\ |\ I(f,p)=+1,\, u(p)>0\}=\left( \sigma(\Theta)+\sigma(\Psi) +\sigma(\Phi_1)+\sigma(\Phi_2)   \right)/4,$$
$$\# \{p\in S_{1,1,1}(f)\ |\ I(f,p)=-1,\, u(p)>0\}=\left( \sigma(\Theta)-\sigma(\Psi)+\sigma(\Phi_1)-\sigma(\Phi_2)    \right)/4.$$
\end{theorem}
\begin{ex} Let $ f=(-y-2z-xy-xz,-2x-2y+3xy+z^2,z+2y-x^2)$, and $u=9-x^2-y^2-z^2$.
 Applying methods presented in Theorem \ref{liczenie1}, with the help of {\sc Singular} \cite{greueletal},
one may check that $I_1=I_2=I_3=\R[x,y,z]$, $\dim_{\R}\Aa=7$, $\sigma(\Theta)=3$,  the quadratic forms $\Psi,\Phi_1,\Phi_2$ are non-degenerate and
$\sigma(\Psi)=1$, $\sigma(\Phi_1)=1$, $\sigma(\Phi_2)=3$. 

Hence $S_{1,1,1}(f)$ consists of three swallowtails, two of them are positive and one is negative. In the ball $\{u>0\}$ there are two positive swallowtails.
\end{ex}


Justyna BOBOWIK\\
Institute of Mathematics, University of Gda\'nsk\\
80-952 Gda\'nsk, Wita Stwosza 57, Poland\\
Justyna.Bobowik@mat.ug.edu.pl\\[1em]
Zbigniew SZAFRANIEC\\
Institute of Mathematics, University of Gda\'nsk\\
80-952 Gda\'nsk, Wita Stwosza 57, Poland\\
Zbigniew.Szafraniec@mat.ug.edu.pl\\

\end{document}